\documentclass[12pt]{article}
\usepackage{amssymb}
\usepackage{amsfonts}
\usepackage{amscd}
\usepackage{mathrsfs}
\usepackage[reqno]{amsmath}
\usepackage{amsthm,upref,amscd}
\textwidth 155mm \textheight 230mm \theoremstyle{plain}

\newtheorem{Thm}{Theorem}[section]

\newtheorem{Lem}[Thm]{Lemma}
\newtheorem{Cor}[Thm]{Corollary}
\newtheorem{Prop}[Thm]{Proposition}

\theoremstyle{definition}

\numberwithin{equation}{section}

\def\R{\mathbb{R}}
\def\N{\mathbb{N}}
\def\lb{\lambda}
\def\va{\varepsilon}
\def\wt{\widetilde}
\def\cj{{\cal J}}
\def\di{\displaystyle\int}
\def\df{\displaystyle\frac}

\begin{document}
\thispagestyle{empty}

\title{On the existence of positive solutions for a quasilinear Schr\"{o}dinger equation\thanks{Supported
by NSFC (11001008, 11271264)}}
\author{Haidong Liu$^{\rm 1}$,\ \ Leiga Zhao$^{\rm 2}$
\vspace{2mm}\\
{$^{\rm 1}$\small College of Mathematics, Physics and Information Engineering, Jiaxing University}\\
{\small Zhejiang 314001, P.R. China}\\
{$^{\rm 2}$\small Department of Mathematics, Beijing University of Chemical Technology}\\
{\small Beijing 100029, P.R. China}}
\date{}
\maketitle

\begin{abstract}
This paper is concerned with the quasilinear Schr\"{o}dinger equation
\begin{equation*}
-\Delta u+V(x)u- \Delta(u^2)u =h(u),   \ \ \mbox{in} \  \R^N,
\end{equation*}
where $N\geq 3$. Under appropriate assumptions on $V$ and $h$, we establish the existence of
positive solutions. The main novelty is that, unlike most other papers on this problem, we do
not assume $h$ is 4-superlinear at infinity.



\end{abstract}

\section{Introduction and main results}

\hspace*{\parindent}In this paper, we consider the quasilinear
Schr\"{o}dinger equation
\begin{equation}\label{eq1.1}
-\Delta u+V(x)u- \Delta(u^2)u =|u|^{p-2}u, \quad\mbox{in} \  \R^N,
\end{equation}
where $N\geq 3,\ 2<p<2\cdot 2^*$,\ $2^*=\frac{2N}{N-2}$ is the critical Sobolev exponent and
$V$ is a continuous function. It is known that, via the ansatz $z(t,x)=e^{-iEt}u(x)$, solutions
of problem \eqref{eq1.1} correspond to stationary waves of
\begin{equation*}
i\partial_tz=-\Delta z+W(x)z- \Delta(|z|^2)z -|z|^{p-2}z, \quad\mbox{in} \  \R\times\R^N,
\end{equation*}
where $W=V+E$ is a new potential. Quasilinear Schr\"{o}dinger equations of this type arise
in plasma physics, see e.g. \cite{K,LSS} for details on the physical background.

The natural energy functional corresponding to \eqref{eq1.1} is given by
$$\Phi(u)=\df{1}{2}\di_{\R^N}(|\nabla u|^2+V(x)u^2)\ dx+\di_{\R^N}u^2|\nabla u|^2\ dx
-\df{1}{p}\di_{\R^N}|u|^p\ dx,$$
which is not well defined in $H^1(\R^N)$. Due to this fact, the
usual variational methods can not be applied directly. This
difficulty makes problems like \eqref{eq1.1} interesting and
challenging. Indeed, during the last ten years, there have been a
considerable amount of researches on such problems. Many existence
and multiplicity results were proved by different approaches, such
as minimizations \cite{LW,PSW}, change of variables
\cite{CJ,FS,LWW}, Nehari method \cite{LWW2} and perturbation method
\cite{LLW}. In \cite{LWW}, by a suitable change of variables, the
quasilinear problem \eqref{eq1.1} was reduced to a semilinear one
and existence reults were given in the cases of bounded, radial or
coercive potential in an Orlicz space framework. By similar change
of variables, Colin and Jeanjean \cite{CJ} investigated the new
functional in $H^1(\R^N)$. They established the existence of
solutions of (\ref{eq1.1}) with $V(x)\equiv 1$ and the general
nonlinearity introduced by Berestycki and Lions \cite{BL}. Moreover,
under the following variant Ambrosetti-Rabinowitz condition
\begin{enumerate}
\item[$(AR)$] there exists $\mu>4$ such that $0<\mu \di_0^t h(s)\ ds\leq h(t)t$ for all $t\in\R^+$,
\end{enumerate}
the existence result was also obtained for the well potential and
the power nonlinearity  $|u|^{p-2}u$ replaced by $h$.

It is worth pointing out that most of these results are based on the
condition $4\leq p<2\cdot 2^*$. As observed in \cite{LWW2}, the
number $2\cdot 2^*$ behaves as a critical exponent for problem
\eqref{eq1.1}. In fact, nonexistence result can be formulated when
$p\geq 2\cdot 2^*$ by a Pohozaev type identity.

To the best of our knowledge, very few results are known about
problem \eqref{eq1.1} with $p\in (2,4)$. We are only aware of the
papers \cite{CJ,G,LW,PSW,RS}. In \cite{LW,PSW}, an unknown Lagrange
multiplier appears in the equation. In \cite{CJ}, an existence
result was given for constant potential. In \cite{G}, semiclassical
solution of \eqref{eq1.1} was studied and it was shown that there
exists a positive solution which concentrates at a local minimum of
the potential.  Recently, Ruiz and Siciliano  \cite{RS} proved the
existence of a positive ground state solution of \eqref{eq1.1} with
$2<p<2\cdot 2^*$. The proof is based on a constrained minimization
procedure. We remark that a concavity hypothesis was imposed on the
potential, which is technique and important in their arguments. For
more references related to \eqref{eq1.1}, we refer the reader to
\cite{DS,DMS,GT,LW2,LWG,LLW2,SV,WZ}.

Inspired by \cite{CJ,RS}, we consider the following equation
\begin{equation}\label{eq1.2}
-\Delta u+V(x)u- \Delta(u^2)u =h(u),   \ \ \mbox{in} \  \R^N,
\end{equation}
where $N\geq 3$ and $V\in C^1(\R^N,\R)$ satisfies the following conditions:
\begin{enumerate}
\item[$(V_1)$] $V(x)=V(|x|)$ and $0<V_0\leq V(x)\leq V_1<\infty$ for all $x\in \R^N$;
\item[$(V'_1)$] $0<V_0\leq V(x)\leq V_\infty:=\lim\limits_{|x|\to \infty} V(x)<\infty$ for all $x\in\R^N$;
\item[$(V_2)$] there exists $\alpha\in [1,2)$ such that
$$\left(\nabla V(x)\cdot x\right)^+\in L^{\frac{2^*}{2^*-\alpha}}(\R^N),$$
where $\left(\nabla V(x)\cdot x\right)^+=\max\{\nabla V(x)\cdot
x,0\}$.
\end{enumerate}
For the nonlinearity $h$, we assume:
\begin{enumerate}
\item[$(h_1)$] $h\in C(\R,\R)$,\ $h(t)=0$ for $t\leq 0$ and $h(t)=o(t)$ as $t\to 0^+$;
\item[$(h_2)$] there exists $C>0$ and $q\in \left(2,2\cdot 2^*\right)$ such that
    $$|h(t)|\leq C(t+t^{q-1}),\quad \mbox{for all}\ t\in \R^+;$$
\item[$(h_3)$] $\lim\limits_{t\to+\infty}\frac{h(t)}{t}=+\infty$;
\item[$(h_4)$] $H(t)=\int_0^t h(s) ds\geq 0$ for all $t\in\R^+$.
\end{enumerate}

The first two results of this paper are the following theorems.

\begin{Thm} \label{thm1}
Suppose that $(V_1),\ (V_2)$ and $(h_1)-(h_4)$ hold. Then problem \eqref{eq1.2} has at least
a positive solution.
\end{Thm}

\begin{Thm} \label{thm2}
Suppose that $(V'_1),\ (V_2)$ and $(h_1)-(h_4)$ hold. Then problem \eqref{eq1.2} has at least
a positive solution.
\end{Thm}

\noindent{\bf Remark 1.1.} {(1) Compared with \cite{RS}, the current paper deals with a more
general nonlinearity and conditions on the potential are different. Especially, we do not need
concavity hypothesis which is essential in their arguments. Thus our results can be regarded
as complements of Theorem 1.1 in \cite{RS}.\\
(2) Condition $(V_2)$ shall be used to prove the boundedness of a special Palais-Smale sequence.
Similar conditions can be found in \cite{AP,JT}. It should be mentioned that, due to the well
properties of the transformation (see Lemma \ref{lem1} and Corollary \ref{cor1}), we only
need a weaker condition than the one in \cite{JT}.\\
(3) We point out that $(h_4)$ is assumed for the sake of simplicity
and it can be dropped. In fact, by $(h_1)$ and $(h_3)$, we can
choose a large positive constant $M$ such that $h(t)+Mt\geq 0$ for
all $t\in\R^+$. Then, applying the arguments in this paper to the
equation
\begin{equation*}
-\Delta u+(V(x)+M)u- \Delta(u^2)u =h(u)+Mu,   \ \ \mbox{in} \  \R^N,
\end{equation*}
we obtain a positive solution of \eqref{eq1.2}.\\
(4) In Section 6, we apply our methods to problem \eqref{eq1.2} with
a nonlinearity of Berestycki and Lions type \cite{BL}, see Theorem
\ref{thm4}.}

\vskip 3mm

The second part of this paper is devoted to problem \eqref{eq1.2} with a parameter
\begin{equation}\label{eq1.3}
-\Delta u+V(x)u- \Delta(u^2)u =\mu h(u),   \ \ \mbox{in} \  \R^N,
\end{equation}
where $N\geq 3$ and $V\in C^1(\R^N,\R)$ satisfies $(V_1)$ and
\begin{enumerate}
\item[$(V'_2)$] $\nabla V(x)\cdot x\leq 0$ for all $x\in \R^N$.
\end{enumerate}
For the nonlinearity $h$, we only assume the following conditions near the origin:
\begin{enumerate}
\item[$(h'_1)$] $h\in C(\R,\R)$,\ $h(t)=0$ for $t\leq 0$ and there exists $q\in \left(2,2^*\right)$ such that
$$\limsup\limits_{t\to 0^+}\frac{h(t)t}{t^{q}}<+\infty;$$
\item[$(h'_2)$] there exists $p\in\left(2,2^*\right)$ such that
$$\liminf\limits_{t\to 0^+} \frac{H(t)}{t^{p}}>0.$$
\end{enumerate}

\begin{Thm} \label{thm3}
Suppose that $(V_1),\ (V'_2)$ and $(h'_1)-(h'_2)$ hold. If
\begin{equation}\label{eq1.4}
p-q<\df{2(p-2)(2^*-p)}{2^*(2^*-2)-2(p-2)},
\end{equation}
then there exists $\mu_0>0$ such that, for any $\mu>\mu_0$, problem \eqref{eq1.3} has at least
a positive solution.
\end{Thm}

\noindent{\bf Remark 1.2.} {(1) From $(h'_1)$ and $(h'_2)$, it is easy to see that $q\leq p$.
Moreover, assumption \eqref{eq1.4} holds if $q$ is close to $p$.\\
(2) In Theorem \ref{thm3}, there is no condition assumed on the
nonlinearity near infinity. Similar assumptions were used in
\cite{chen-li,CW} for the semilinear  elliptic problems on a bounded
domain. }

\vskip 3mm

To prove Theorems \ref{thm1} and \ref{thm2}, we are faced with several difficulties. On one
hand, due to the presence of $\Delta(u^2)u$ and growth condition on the nonlinearity, the
natural energy functional is not well defined in $H^1(\R^N)$. Thus we can not apply variational
methods directly. To overcome this difficulty, we employ an argument developed in \cite{CJ,LWW}
and make a change of variables to reformulate the problem into a semilinear one.

On the other hand, it will be shown later that the functional $I$ associated to equivalent
semilinear problem possesses the Mountain Pass geometry (see Lemma \ref{lem2}) and so there
exists a Palais-Smale sequence. However, the boundedness of Palais-Smale sequence seems hard
to verify. Our strategy is applying Jeanjean's monotonicity trick \cite{J}, which can be traced
back to \cite{S}, to find a bounded Palais-Smale sequence for $I$. More precisely, we will
take the following three steps. Firstly, we define a family of functionals $I_\lb,\ \lb\in[\frac{1}{2},1]$,
such that $I_1=I$.  By an abstract result in \cite{J}, for almost every $\lb\in[\frac{1}{2},1]$,
there is a bounded Palais-Smale sequence for $I_\lb$. Secondly, restricting in the subspace
of radially symmetric functions if $V$ is a radial potential or using a version of global
compactness lemma due to Adachi and Watanabe \cite{AW} when $V$ is a well potential, we obtain
a nontrivial critical point $v_\lb$ of $I_\lb$ for almost every $\lb\in[\frac{1}{2},1]$. Finally,
choosing $\lb_n\to 1$, we have a sequence of $\{v_{\lb_n}\}$ being the critical points of
$I_{\lb_n}$. Then, with the aid of a corresponding Pohozaev type identity and condition $(V_2)$,
we prove that $\{v_{\lb_n}\}$ is a bounded Palais-Smale sequence for $I$.

The proof of Theorem \ref{thm3} is based on the conclusion of Theorem \ref{thm1} and a priori
estimate. Firstly, we modify $h$ to a new nonlinearity $\wt{h}$ which satisfies $(h_1)-(h_4)$.
In view of Theorem \ref{thm1}, the modified problem has a positive solution. Secondly, it will
be shown that the solution obtained converges to zero in $L^\infty$-norm as $\mu\to\infty$.
Thus, for $\mu$ large, it is a solution of original problem. This method is borrowed from Costa
and Wang \cite{CW}. But here we have to analyse carefully the effect of the term $\Delta(u^2)u$
and the transformation $f$.

The paper is organized as follows. In Section 2, following the method in \cite{CJ,LWW}, we
reformulate (\ref{eq1.2}) into a semilinear problem. We give the proofs of Theorems \ref{thm1}$-$\ref{thm3}
in Sections 3$-$5 respectively. The last section is devoted to a generalized result.

\textbf{Notations:} In the sequel, $C$ and $C_i$ represent variant positive constants. The
standard norms of $L^p(\R^N)\ (p\geq 1)$ and $H^1(\R^N)$ are denoted by $|\cdot|_p$ and $\|\cdot\|$
respectively. Set $H_r^1(\R^N)=\left\{u\in H^1(\R^N)\ |\ u\ \hbox{is radially symmetric}\right\}.$

\section{Equivalent variational problem}

\hspace*{\parindent} The natural energy functional associated to \eqref{eq1.2} is
$$J(u)=\df{1}{2}\di_{\R^N}(|\nabla u|^2+V(x)u^2)\ dx + \di_{\R^N}u^2|\nabla u|^2\ dx-\di_{\R^N}H(u)\ dx,$$
which is not well defined for all $u\in H^1(\R^N)$. To apply variational methods, we employ
an argument developed in \cite{CJ,LWW} and make a change of variables.

Let $f$ be defined by
$$ f'(t)=\df{1}{\sqrt{1+2f^2(t)}} \quad \hbox{and}\quad f(0)=0$$
on $[0,+\infty)$ and by $f(t)=-f(-t)$ on $(-\infty, 0]$. Then $f$ is uniquely defined, smooth
and invertible. In next lemma, we summarize some properties of $f$ which have been proved in
\cite{CJ,LWW}.

\begin{Lem} \label{lem1}
(1) $|f'(t)|\leq 1$ for all $t\in\R$;\\
(2) $|f(t)|\leq |t|$ for all $t\in\R$;\\
(3) $|f(t)|\leq 2^{\frac{1}{4}}|t|^{\frac{1}{2}}$ for all $t\in\R$;\\
(4) $\lim\limits_{t\to 0}\frac{f(t)}{t}=1$;\\
(5) $\lim\limits_{t\to +\infty}\frac{f(t)}{\sqrt{t}}=2^{\frac{1}{4}}$;\\
(6) $\frac{1}{2} f(t)\leq tf'(t)\leq f(t)$ for all $t\in\R^+$;\\
(7) $\frac{1}{2} f^2(t)\leq f(t)f'(t)t\leq f^2(t)$ for all $t\in \R$;\\
(8) there exists a positive constant $C$ such that
\begin{equation*}
|f(t)|\geq\left\{ \begin{aligned}
& C|t|, \quad\ \ \mbox{if}\ |t|\leq 1, \\
& C|t|^{\frac{1}{2}}, \quad\mbox{if}\ |t|\geq 1.
\end{aligned} \right.
\end{equation*}
\end{Lem}

As a consequence of Lemma \ref{lem1}, we have

\begin{Cor}\label{cor1}
Let $\alpha\in[1,2)$, then $|f(t)|\leq 2^{\frac{1}{4}}|t|^{\frac{\alpha}{2}}$ for all $t\in\R$.
\end{Cor}

Set $v=f^{-1}(u)$, then we obtain
$$I(v):=J(f(v))=\df{1}{2}\di_{\R^N} \left(|\nabla v|^2+V(x) f^2(v) \right) \ dx -\di_{\R^N} H(f(v))\ dx,$$
which is well defined in $H^1(\R^N)$ and belongs to $C^1$ under our assumptions. It is well
known that critical points of $I$ are weak solutions of semilinear elliptic equation
$$-\Delta v=h(f(v))f'(v)-V(x)f(v)f'(v),\quad\hbox{in}\ \R^N.$$
Moreover, if $v\in H^1(\R^N)\cap C^2(\R^N)$ is a critical point of $I$, then $u=f(v)$ is a
classical solution of problem \eqref{eq1.2}.

\section{Proof of Theorem \ref{thm1}}

\hspace*{\parindent}
As stated in the introduction, since we do not assume $h$ is 4-superlinear at infinity, it
seems hard to prove the boundedness of Palais-Smale sequence. We will use the following abstract
result \cite{J} to construct a special Palais-Smale sequence.

\begin{Prop} \label{prop1}
Let $X$ be a Banach space equipped with a norm $\|\cdot\|_X$ and let $\cj\subset\R^+$ be an
interval. We consider a family $\{\Phi_\lb\}_{\lb\in\cj}$ of $C^1$-functionals on $X$ of the
form
$$\Phi_\lb(v)=A(v)-\lb B(v),\quad\hbox{for all}\ \lb\in\cj,$$
where $B(v)\geq 0$ for all $v\in X$ and either $A(v)\to+\infty$ or $B(v)\to+\infty$ as $\|v\|_X\to\infty$.
Assume that there exist two points $v_1,\ v_2\in X$ such that
$$c_\lb=\inf\limits_{\gamma\in\Gamma}\max\limits_{t\in [0,1]}\Phi_\lb(\gamma(t))
>\max\{\Phi_\lb(v_1),\Phi_\lb(v_2)\},\quad\hbox{for all}\ \lb\in\cj,$$
where $\Gamma=\{\gamma\in C([0,1],X)\ |\ \gamma(0)=v_1,\ \gamma(1)=v_2\}.$ Then, for almost
every $\lb\in\cj$, there exists a sequence $\{v_n(\lb)\}\subset X$ such that\\
(1)\ $\{v_n(\lb)\}$ is bounded in $X$,\\
(2)\ $\Phi_\lb(v_n(\lb))\to c_\lb$,\\
(3)\ $\Phi'_\lb(v_n(\lb))\to 0$ in $X^*$, where $X^*$ is the dual space of $X$.\\
Furthermore, the map $\lb\mapsto c_\lb$ is continuous from the left.
\end{Prop}

To apply Proposition \ref{prop1}, we set $X=H_r^1(\R^N)$ and introduce a family of functionals
$$I_\lb(v)=\df{1}{2}\di_{\R^N} \left(|\nabla v|^2+V(x) f^2(v) \right)\ dx -\lb\di_{\R^N} H(f(v))\ dx,\quad v\in X,$$
where $\lb\in[\frac{1}{2},1]$.

Define $A(v)=\frac{1}{2}\int_{\R^N} \left(|\nabla v|^2+V(x) f^2(v) \right)\ dx$ and
$B(v)=\int_{\R^N} H(f(v))\ dx.$ Then $I_\lb(v)=A(v)-\lb B(v)$. Next lemma ensures that $I_\lb$
satisfies all assumptions of Proposition \ref{prop1}.

\begin{Lem} \label{lem2}
Assume that $(V_1)$ and $(h_1)-(h_4)$ hold. Then\\
(1) $B(v)\geq 0$ for all $v\in X$;\\
(2) $A(v)\to \infty$ as $\|v\|\to\infty$;\\
(3) there exists $v_0\in X$, independent of $\lb$, such that $I_\lb(v_0)< 0$ for all $\lb\in[\frac{1}{2},1]$;\\
(4) for all $\lb\in[\frac{1}{2},1]$, it holds
$$c_\lb=\inf_{\gamma\in\Gamma}\max_{t\in [0,1]}I_\lb(\gamma(t))>\max\{I_\lb(0),I_\lb(v_0)\},$$
\quad\ \ where $\Gamma=\{\gamma\in C([0,1],X)\ |\ \gamma(0)=0,\ \gamma(1)=v_0\}.$
\end{Lem}

\begin{proof}
(1) is a direct consequence of $(h_4)$. Now we prove (2). By Lemma \ref{lem1}, we deduce
\begin{equation*}
\begin{array}{rl}
\|v\|^2
&=\di_{\R^N}|\nabla v|^2 \ dx+\di_{\{x | |v(x)|\leq 1\}} v^2 \ dx+\di_{\{x | |v(x)|> 1\}} v^2 \ dx\\[5mm]
&\leq \di_{\R^N}|\nabla v|^2 \ dx+C\di_{\{x | |v(x)|\leq 1\}} f^2(v)\ dx+\di_{\{x | |v(x)|> 1\}} |v|^{2^*}\ dx\\[5mm]
&\leq \di_{\R^N}|\nabla v|^2 \ dx+C_1\di_{\R^N} V(x)f^2(v)\ dx
+C_2\left(\di_{\R^N}|\nabla v|^2\ dx\right)^{\frac{2^*}{2}}\\[5mm]
&\leq C_3 \left(A(v)+A(v)^{\frac{2^*}{2}}\right),
\end{array}
\end{equation*}
which implies the coercivity of $A$.

In order to prove (3), we set
$$J_\lb(u)=\df{1}{2}\di_{\R^N}(|\nabla u|^2+V(x)u^2)\ dx + \di_{\R^N}u^2|\nabla u|^2\ dx-\lb \di_{\R^N} H(u)\ dx.$$
Let us fix some nonnegative radially symmetric function $u\in C_0^\infty(\R^N)\setminus\{0\}$.
Then, for $t>0$, we have
\begin{equation*}
\begin{array}{rl}
J_{1/2}(tu(x/t)) & = \df{t^N}{2}\di_{\R^N} |\nabla u|^2 \ dx +\df{t^{N+2}}{2}\di_{\R^N} V(tx) u^2 \ dx\\[5mm]
&\qquad\qquad\qquad\qquad\ +\ t^{N+2}\di_{\R^N} u^2 |\nabla u|^2 \ dx-\df{t^N}{2} \di_{\R^N} H(tu) \ dx\\[5mm]
&\leq \df{t^{N+2}}{2}\bigg[\df{1}{t^2}\di_{\R^N} |\nabla u|^2 \ dx+ \di_{\R^N} V_1 u^2 \ dx\\[5mm]
&\qquad\qquad\qquad\qquad\qquad\ \ +\ 2 \di_{\R^N} u^2 |\nabla u|^2 \ dx-\di_{\R^N} \df{H(tu)}{t^2} \ dx\bigg].
\end{array}
\end{equation*}
By assumption $(h_3)$, it is easy to see that $J_{1/2}(tu(x/t))<0$ for $t$ large. Thus there
exists $v_0=f^{-1}(u_0)\in X$ (independent of $\lb\in[\frac{1}{2},1]$) such that
$I_\lb(v_0)=J_\lb(u_0)\leq J_{1/2}(u_0)<0$ for all $\lb\in[\frac{1}{2},1]$.

It remains to prove (4). Define $\widehat{H}(t)=-\frac{V_0}{2}f^2(t)+H(f(t))$. Using $(h_1),\ (h_2)$
and Lemma \ref{lem1}, we obtain
$$\lim\limits_{t\to 0}\df{\widehat{H}(t)}{t^2}=-\df{V_0}{2}\quad\hbox{and}\quad \lim\limits_{t\to\infty}\df{\widehat{H}(t)}{|t|^{2^*}}=0.$$
Thus there exists $C>0$ such that
$$\widehat{H}(t)\leq -\df{V_0}{4} t^2+ C |t|^{2^*},\quad\hbox{for all}\ t\in\R.$$
It follows that
\begin{equation*}
\begin{array}{rl}
I_\lb(v) & \geq \df{1}{2}\di_{\R^N} |\nabla v|^2\ dx +\df{1}{2}\di_{\R^N} V_0 f^2(v)\ dx
-\di_{\R^N} H(f(v))\ dx\\[5mm]
& \geq \df{1}{2}\di_{\R^N} |\nabla v|^2\ dx + \df{V_0}{4}\di_{\R^N} v^2 \ dx -C \di_{\R^N} |v|^{2^*}\ dx\\[5mm]
& \geq \min\{\frac{1}{2},\frac{V_0}{4}\}\|v\|^2-C\|v\|^{2^*}.
\end{array}
\end{equation*}
From this, we get $c_\lb>0$ and the proof is complete.
\end{proof}

By Lemma \ref{lem2} and Proposition \ref{prop1}, there exists $\cj_1\subset[\frac{1}{2},1]$
with $meas(\cj_1)=0$ such that, for any $\lb\in[\frac{1}{2},1]\setminus\cj_1$, there is a
sequence $\{v_n\}\subset X$ satisfying
$$(i)\ \{v_n\}\ \hbox{is bounded in}\ X,\quad (ii)\ I_\lb(v_n)\to c_\lb,\quad
(iii)\ I'_\lb(v_n)\to 0\ \hbox{in}\ X^*.$$

\begin{Lem} \label{lem3}
Up to a subsequence, $\{v_n\}$ converges to a positive critical point $v_\lb$ of $I_\lb$
with $I_\lb(v_\lb)= c_\lb$.
\end{Lem}

\begin{proof}
Without loss of generality, we can suppose that $q\in (4,2\cdot2^*)$ in condition $(h_2)$.
Since $\{v_n\}\subset X$ is bounded, up to a subsequence, we have
$$v_n\rightharpoonup v_\lb\ \hbox{in}\ X, \quad v_n\to v_\lb\ \hbox{in}\ L^{\frac{q}{2}}(\R^N),
\quad v_n\to v_\lb\ \hbox{a.e. in}\ \R^N,$$
for some $v_\lb\in X$. It is easy to check that $I'_\lb(v_\lb)=0$. Next we prove $v_n\to v_\lb$
in $X$. First of all, setting $G(x,t)=\frac{1}{2}V(x)t^2-\frac{1}{2}V(x)f^2(t)+\lb H(f(t))$,
we can rewrite $I_\lb$ as
$$I_\lb(v)=\df{1}{2}\di_{\R^N}\left(|\nabla v|^2+V(x)v^2\right)\ dx-\di_{\R^N}G(x,v)\ dx.$$
Let $g(x,t)=\frac{d}{dt}G(x,t)$. By $(h_1),\ (h_2)$ and Lemma \ref{lem1}, for any $\va>0$,
there exists $C(\va)>0$ such that
$$|g(x,t)|\leq \va|t|+C(\va)|t|^{\frac{q-2}{2}},\quad\hbox{for all}\ x\in\R^N\ \hbox{and}\ t\in\R.$$
Using this inequality and the fact $v_n\to v_\lb\ \hbox{in}\ L^{\frac{q}{2}}(\R^N)$, we obtain
$$\lim\limits_{n\to\infty}\di_{\R^N}(g(x,v_n)-g(x,v_\lb))(v_n-v_\lb)\ dx=0.$$
Hence
\begin{equation*}
\begin{array}{rl}
o(1)&=\langle I'_\lb(v_n)-I'_\lb(v_\lb), v_n-v_\lb \rangle\\[5mm]
&=\di_{\R^N} \left(|\nabla (v_n-v_\lb)|^2+V(x)(v_n-v_\lb)^2\right)\ dx-\di_{\R^N}\left(g(x,v_n)-g(x,v_\lb)\right)(v_n-v_\lb)\ dx\\[5mm]
&\geq  \min\{1,V_0\}\|v_n-v_\lb\|^2+o(1),
\end{array}
\end{equation*}
which implies $v_n\to v_\lb$ in $X$. Therefore $v_\lb$ is a nontrivial critical point of $I_\lb$
with $I(v_\lb)=c_\lb$. The positivity of $v_\lb$ follows by a standard argument.
\end{proof}

At this point, for almost every $\lb\in[\frac{1}{2},1]$, we obtain a positive critical point
$v_\lb$ of $I_\lb$. In general, it is not known whether it is true for $\lb=1$. However we have

\begin{Lem} \label{lem4}
Under the assumptions of Theorem \ref{thm1}, there exist $\{\lb_n\}\subset[\frac{1}{2},1]$ and
$\{v_n\}\subset X\setminus\{0\}$ such that $\lim\limits_{n\to\infty}\lb_n=1,\ v_n>0,\
I_{\lb_n}(v_n)=c_{\lb_n}\leq c_{1/2}$ and $I'_{\lb_n}(v_n)=0$.
\end{Lem}

Next we show that the sequence $\{v_n\}$ obtained in Lemma \ref{lem4} is bounded. For this
purpose, we shall use the following Pohozaev type identity. Since the proof is standard, we
omit it.

\begin{Lem} \label{lem5}
If $v\in X$ is a critical point of $I_\lb$, then
\begin{equation*}
\begin{array}{rl}
\df{N-2}{2} \di_{\R^N} |\nabla v|^2 \ dx &+\ \df{N}{2} \di_{\R^N} V(x) f^2(v) \ dx\\[5mm]
&+\ \df{1}{2} \di_{\R^N}   \nabla V(x)\cdot x f^2(v) \ dx -\lb N\di_{\R^N} H(f(v))\ dx=0.
\end{array}
\end{equation*}
\end{Lem}

\begin{Lem} \label{lem6}
The sequence $\{v_n\}$ obtained in Lemma \ref{lem4} is bounded in $X$.
\end{Lem}

\begin{proof}
In view of Lemma \ref{lem2}, it is enough to prove that $\int_{\R^N}\left(|\nabla v_n|^2+V(x) f^2(v_n) \right)\ dx$
is bounded. By $I_{\lb_n}(v_n)\leq c_{1/2}$, Lemma \ref{lem5}, H\"{o}lder inequality and
Sobolev inequality, we have
\begin{equation} \label{eq3.1}
\begin{array}{rl}
\di_{\R^N} |\nabla v_n|^2 \ dx &\leq  \df{1}{2}\di_{\R^N} \nabla V(x)\cdot x f^2(v_n) \ dx+N c_{1/2}\\[5mm]
&\leq \df{1}{2}\ \big| (\nabla V(x)\cdot x)^+\big|_{\frac{2^*}{2^*-\alpha}}
\left(\di_{\R^N} f^{\frac{2\cdot 2^*}{\alpha}}(v_n) \ dx\right)^{\frac{\alpha}{2^*}}+N c_{1/2}\\[5mm]
&\leq C\left(\di_{\R^N}|v_n|^{2^*}\ dx\right)^{\frac{\alpha}{2^*}}+N c_{1/2}\\[5mm]
&\leq C\left(\di_{\R^N}|\nabla v_n|^2\ dx\right)^{\frac{\alpha}{2}}+N c_{1/2},
\end{array}
\end{equation}
where we used assumption $(V_2)$ and Corollary \ref{cor1}. Since $\alpha\in[1,2)$, we obtain
the boundedness of $\int_{\R^N} |\nabla v_n|^2 \ dx$.

Next we prove that $\int_{\R^N} V(x) f^2(v_n) \ dx$ is bounded. By $(h_1),\ (h_2)$ and Lemma
\ref{lem1}, we get
$$\lim\limits_{t\to 0}\df{|h(f(t))f'(t)t|}{f^2(t)}=0\quad\hbox{and}\quad
\lim\limits_{t\to\infty}\df{|h(f(t))f'(t)t|}{|t|^{2^*}}=0.$$
Thus, for any $\va>0$, there exists $C(\va)>0$ such that
\begin{equation}\label{eq3.2}
|h(f(t))f'(t)t|\leq \va f^2(t)+C(\va)|t|^{2^*},\quad\hbox{for all}\ t\in\R.
\end{equation}
Then we have, using $\langle I'_{\lb_n}(v_n),v_n \rangle=0$ and Lemma \ref{lem1},
\begin{equation*}
\begin{array}{rl}
&\quad\di_{\R^N} |\nabla v_n|^2 \ dx  + \df{1}{2}\di_{\R^N}V(x) f^2(v_n) \ dx\\[5mm]
&\leq\di_{\R^N} |\nabla v_n|^2 \ dx  + \di_{\R^N}V(x) f(v_n)f'(v_n)v_n \ dx \\[5mm]
&=\lb_n\di_{\R^N} h(f(v_n)) f'(v_n) v_n \ dx\\[5mm]
&\leq \va \di_{\R^N}  | f(v_n)|^2 \ dx+C(\va)\di_{\R^N}  |v_n|^{2^*} \ dx\\[5mm]
& \leq \df{\va}{V_0}\di_{\R^N} V(x) f^2(v_n)\ dx + C'(\va)\left(\di_{\R^N} |\nabla v_n|^2\ dx\right)^{\frac{2^*}{2}}.
\end{array}
\end{equation*}
Choosing $\va>0$ small enough, we complete the proof.
\end{proof}

\begin{proof}[Proof of Theorem \ref{thm1} (completed)]
By Lemmas \ref{lem4} and \ref{lem6}, there exist $\{\lb_n\}\subset[\frac{1}{2},1]$ and a bounded
sequence $\{v_n\}\subset X\setminus\{0\}$ such that
$$\lim\limits_{n\to\infty}\lb_n=1,\quad I_{\lb_n}(v_n)=c_{\lb_n},\quad I'_{\lb_n}(v_n)=0.$$
Then
$$\lim\limits_{n\to\infty}I(v_n)=\lim\limits_{n\to\infty}\left(I_{\lb_n}(v_n)
+(\lb_n-1)\di_{\R^N} H(f(v_n))\ dx\right)=\lim\limits_{n\to\infty} c_{\lb_n}=c_1,$$
where we used the fact that the map $\lb\mapsto c_\lb$ is continuous from the left. Similarly,
$I'(v_n)\to 0$ in $X^*$. That is, $\{v_n\}$ is a bounded Palais-Smale sequence for $I$ satisfying
$\lim\limits_{n\to\infty}I(v_n)=c_1$. Using Lemma \ref{lem3} again, we obtain a positive critical
point $v$ of $I$.
\end{proof}

\section{Proof of Theorem \ref{thm2}}

\hspace*{\parindent} This section is devoted to the case of well potential. In what follows,
we always assume that $V(x)\not\equiv V_\infty$ (otherwise Theorem \ref{thm1} gives the conclusion).
We first recall some known results of ``limit" functional
$$I_\lb^\infty(v)=\df{1}{2}\di_{\R^N} \left(|\nabla v|^2+V_\infty f^2(v) \right)\ dx -\lb\di_{\R^N} H(f(v)) \ dx.$$
Define
$$m^\infty_\lb=\inf\{I_\lb^\infty(v)\ |\ v\in H^1(\R^N)\setminus \{0\},\ (I_\lb^\infty)'(v)=0\}.$$
The following proposition  \cite{CJ} presents the results on least energy solutions for autonomous
problems which are crucial to ensure the compactness of bounded Palais-Smale sequences.

\begin{Prop} \label{prop2}
Under assumptions $(h_1)-(h_3),\ m^\infty_\lb>0$ and is achieved by some positive function
$w_\lb^\infty\in H^1(\R^N)$. Moreover, we can find a path $\gamma\in C([0,1], H^1(\R^N))$
such that $\gamma(t)(x)>0$ for all $x\in\R^N$ and $t\in (0,1]$,\ $\gamma(0)=0, \ I_\lb^\infty(\gamma(1))<0$,\
$w_\lb^\infty\in \gamma([0,1])$ and
$$\max_{t\in[0,1]} I_\lb^\infty(\gamma(t))=  I_\lb^\infty(w_\lb^\infty).$$
\end{Prop}

\vskip3truemm

\begin{Lem} \label{lem7}
Assume $(V'_1)$ and $(h_1)-(h_3)$ hold. Define
$c_\lb=\inf\limits_{\gamma\in\Gamma}\max\limits_{t\in [0,1]}I_\lb(\gamma(t)),$ where $I_\lb$
is given in Section 3 and $\Gamma=\{\gamma\in C([0,1],H^1(\R^N))\ |\ \gamma(0)=0,\ I_\lb(\gamma(1))<0\}.$
Then $c_\lb <m_\lb^\infty$ for any $\lb\in[\frac{1}{2},1]$.
\end{Lem}

\begin{proof}
Let $w_\lb^\infty$ and $\gamma$ be chosen as in Proposition \ref{prop2}. Then
$$I_\lb(\gamma(t))<I_\lb^\infty(\gamma(t)),\ \ \hbox{for all}\ t\in (0,1],$$
and it follows that
$$c_\lb\leq \max_{t\in[0,1]} I_\lb(\gamma(t))<\max_{t\in[0,1]} I_\lb^\infty( \gamma(t))= m_\lb^\infty,$$
which completes the proof.
\end{proof}

Since $m^\infty_\lb>0$, We have the following decomposition of bounded Palais-Smale sequences,
which was proved in \cite{AW}.

\begin{Prop} \label{prop3}
Suppose that $(V'_1)$ and $(h_1)-(h_2)$ are satisfied. Let $\{v_n\}\subset H^1(\R^N)$ be a
bounded Palais-Smale sequence for $I_\lb$. Then there exists a subsequence of $\{v_n\}$,
denoted also by $\{v_n\}$, an integer $l\in \N \cup \{0\}$, sequences $\{y^k_n\}\subset \R^N$,
$w^k\in H^1({\R^N})$ for $1\leq k\leq l$, such that\\[2mm]
(1)\ $|y^k_n|\rightarrow \infty$ and $|y^k_n-y^{k'}_n|\rightarrow \infty$ as $n\to\infty$, for $k\not=k'$,\\[2mm]
(2)\ $v_n\rightharpoonup v_0$ in $H^1(\R^N)$ with ${I_\lb'}(v_0)=0$,\\[2mm]
(3)\ $w^k\not=0$ and $(I_\lb^\infty)'(w^k)=0$ for $1\leq k \leq l$,\\[2mm]
(4)\ $\left\|v_n-v_0-\sum\limits_{k=1}^{l} w^k(\cdot-y_n^k)\right\|\rightarrow 0$,\\[1mm]
(5)\ $I_\lb (v_n)\rightarrow I_\lb (v_0)+\sum\limits_{k=1}^{l}I_\lb^\infty(w^k)$,\\[1mm]
where we agree that in the case $l=0$ the above holds without $w^k$ and $\{y_n^k\}$.
\end{Prop}

Using Lemma \ref{lem7} and Proposition \ref{prop3}, we can prove

\begin{Lem} \label{lem8}
Assume that $(V'_1)$ and $(h_1)-(h_3)$ hold. Let $\{v_n\}\subset H^1(\R^N)$ be a bounded
Palais-Smale sequence for $I_\lb$ satisfying $\limsup\limits_{n\to\infty}I_\lb(v_n)\leq c_\lb$
and $\|v_n\|\nrightarrow 0$ as $n\to\infty$. Then, up to a subsequence, $\{v_n\}$ converges
weakly to a nontrivial critical point $v_\lb$ of $I_\lb$ with $I_\lb(v_\lb)\leq c_\lb$.
\end{Lem}

\begin{proof}
By Proposition \ref{prop3}, up to a subsequence, there exist $l\in\N\cup \{0\}$ and
$v_\lb\in H^1({\R^N})$ such that $v_n\rightharpoonup v_\lb$ in $H^1(\R^N),$ \ ${I}_\lb'(v_\lb)=0$
and
$$I_\lb (v_n)\rightarrow I_\lb (v_\lb)+\sum_{k=1}^{l}I_\lb^\infty(w_\lb^k),$$
where $\{w_\lb^k\}_{k=1}^{l}$ are nontrivial critical points of $I_\lb^\infty$.

If $I_\lb(v_\lb)<0$, then the proof is complete. If $I_\lb(v_\lb)\geq 0$, then we claim that
$l=0$. Otherwise,
$$c_\lb\geq\lim_{n\rightarrow \infty}I_\lb(v_n)=I_\lb (v_\lb)
+\sum_{k=1}^{l}I_\lb^\infty(w_\lb^k)\geq m_\lb^\infty,$$
which contradicts Lemma \ref{lem7}.  Thus $v_n\to v_\lb$ in $H^1(\R^N)$ and $I_\lb(v_\lb)\leq c_\lb$.
Since $\|v_n\|\nrightarrow 0$ as $n\to\infty$, $v_\lb$ is a nontrivial critical point of $I_\lb$.
This completes the proof.
\end{proof}

\begin{Lem} \label{lem9}
Under the assumptions of Theorem \ref{thm2}, there exists $\sigma>0$ (independent of $\lb\in [\frac{1}{2},1]$)
such that, if $v$ is a nontrivial critical point of $I_\lb$, then $\|v\|\geq\sigma$.
\end{Lem}

\begin{proof}
It follows from $\langle I'_\lb(v),v\rangle=0$ that
$$\di_{\R^N}\left(|\nabla v|^2 + V(x)f(v)f'(v)v\right)\ dx=\lb\di_{\R^N}h(f(v))f'(v)v \ dx.$$
By Lemma \ref{lem1} and \eqref{eq3.2}, we have
$$\di_{\R^N}|\nabla v|^2 \ dx + \df{1}{2}\di_{\R^N}V(x)f^2(v)\ dx\leq\df{1}{4}\di_{\R^N}V(x)f^2(v)\ dx
+C\di_{\R^N}|v|^{2^*}\ dx,$$
which implies
$$\di_{\R^N}\left(|\nabla v|^2 + V(x)f^2(v)\right)\ dx\leq C\di_{\R^N}|v|^{2^*}\ dx\leq
C\left[\di_{\R^N}\left(|\nabla v|^2 + V(x)f^2(v)\right)\ dx\right]^{\frac{2^*}{2}}.$$
Since $v\neq 0$, we obtain
$$\di_{\R^N}\left(|\nabla v|^2 + V(x)f^2(v)\right)\ dx\geq\sigma_0$$
for some positive constant $\sigma_0$. Then the conclusion follows immediately from $(V'_1)$
and $|f(t)|\leq |t|$ for all $t\in\R$.
\end{proof}

\begin{proof}[Proof of Theorem \ref{thm2}]
It is easy to see that, under hypotheses of Theorem \ref{thm2}, all assumptions in Proposition
\ref{prop1} are satisfied. Then there exists $\cj_1\subset[\frac{1}{2},1]$ with $meas(\cj_1)=0$
such that, for any $\lb\in[\frac{1}{2},1]\setminus\cj_1$, there is a bounded Palais-Smale sequence
$\{v_n\}\subset H^1(\R^N)$ for $I_\lb$ satisfying $\lim\limits_{n\to\infty}I_\lb(v_n)=c_\lb$.
Since $c_\lb>0$, we know $\|v_n\|\nrightarrow 0$ as $n\to\infty$. Using Lemmas \ref{lem8} and
\ref{lem9}, for any $\lb\in[\frac{1}{2},1]\setminus\cj_1$, we obtain a nontrivial critical
point $v_\lb$ of $I_\lb$ with $I_\lb(v_\lb)\leq c_\lb$ and $\|v_\lb\|\geq \sigma>0$.

Choosing $\lb_n\subset[\frac{1}{2},1]\setminus\cj_1$ such that $\lb_n\to 1$ as $n\to\infty$,
we obtain a sequence $\{v_n\}\subset H^1(\R^N)$ satisfying
$$\|v_n\|\geq\sigma>0,\quad I_{\lb_n}(v_n)\leq c_{\lb_n}\leq c_{1/2}, \quad I'_{\lb_n}(v_n)=0.$$
By Lemma \ref{lem6}, $\{v_n\}$ is bounded in $H^1(\R^N)$. Then
$$\limsup\limits_{n\to\infty}I(v_n)=\limsup\limits_{n\to\infty}\left(I_{\lb_n}(v_n)
+(\lb_n-1)\di_{\R^N} H(f(v_n))\ dx\right)\leq\lim\limits_{n\to\infty} c_{\lb_n}=c_1,$$
and $I'(v_n)\to 0$ in $H^{-1}(\R^N)$. That is, $\{v_n\}$ is a bounded Palais-Smale sequence
for $I$ satisfying $\limsup\limits_{n\to\infty}I(v_n)\leq c_1$ and $\|v_n\|\nrightarrow 0$
as $n\to\infty$. Using Lemma \ref{lem8} again, we obtain a nontrivial critical point $v$ of $I$.
A standard argument can show that $v>0$.
\end{proof}

\section{Proof of Theorem \ref{thm3}}

\hspace*{\parindent}  The goal of this section is to prove Theorem \ref{thm3}. To this end,
we use the same idea as in \cite{CW}. Firstly, since there is no assumption on $h$ near infinity,
we need to modify the nonlinearity to a new one which satisfies $(h_1)-(h_4)$. Thanks to Theorem
\ref{thm1}, the modified problem has a positive solution. Secondly, we shall prove that the
solution obtained converges to zero in $L^\infty$-norm as $\mu\to\infty$. Thus, for $\mu$ large,
it is in fact a positive solution of original problem \eqref{eq1.3}.

By $(h'_1)$ and $(h'_2)$, there exist two positive constants $K_0$ and $K_1$ such that
\begin{equation}\label{eq5.1}
h(t) \leq \df{1}{q}K_1 t^{q-1}\quad\hbox{and}\quad K_0 t^{p}\leq H(t) \leq K_1 t^{q}
\end{equation}
for $t>0$ sufficiently small. Choose $\delta>0$ such that \eqref{eq5.1}  holds for $0<t\leq 2\delta$.
Let $\xi$ be a cut-off function satisfying $0\leq\xi\leq 1,\ \xi(t)=1$ for $t\leq\delta$,
\ $\xi(t)=0$ for $t\geq 2\delta$ and $|\xi'(t)|\leq 2/\delta$ for $\delta\leq t\leq 2\delta$.
Define
$$\wt{H}(t)=\xi(t)H(t)+(1-\xi(t))K_1|t|^q$$
and $\wt{h}(t)=\wt{H}'(t)$. Then it is easy to verify that $\wt{h}$ satisfies $(h_1)-(h_4)$.
Moreover, we have

\begin{Lem}\label{lem10}
(1) There exists $C>0$ such that
\begin{equation}\label{eq5.2}
\wt{h}(t)\leq C t^{q-1}, \quad\hbox{for all}\ t>0,
\end{equation}
 and
\begin{equation}\label{eq5.3}
\wt{h}(t)\leq \va t+C\va^{\frac{q-2^*}{q-2}} t^{2^*-1}, \quad\hbox{for all}\ t>0\ \hbox{and}\ \va\in (0,1).
\end{equation}
(2) For any $T>0$, there exists $C(T)>0$ such that
\begin{equation}\label{eq5.4}
\wt{H}(t)\geq C(T) t^p, \quad\hbox{for all}\ t\in [0,T].
\end{equation}
\end{Lem}

Now we consider the modified problem
\begin{equation}\label{eq5.5}
-\Delta u+V(x)u- \Delta(u^2)u =\mu \wt{h}(u),   \ \ \mbox{in} \ \R^N,
\end{equation}
with the natural energy functional given by
$$\wt{J}_\mu(u)=\df{1}{2}\di_{\R^N}(|\nabla u|^2+V(x)u^2)\ dx + \di_{\R^N}u^2|\nabla u|^2\ dx-\mu\di_{\R^N} \wt{H}(u)\ dx.$$
As in Section 2, setting $v=f^{-1}(u)$, we obtain
$$\wt{I}_\mu(v):=\wt{J}_\mu(f(v))=\df{1}{2}\di_{\R^N}\left(|\nabla v|^2+V(x) f^2(v) \right)\ dx
-\mu\di_{\R^N} \wt{H}(f(v))\ dx.$$

In order to show that solutions of the modified problem \eqref{eq5.5} are in fact solutions
of the original problem \eqref{eq1.3}, we need the following $L^\infty$-estimate by Moser
iteration. We give the proof for completeness.

\begin{Lem}\label{lem11}
If $v$ is a positive critical point of $\wt{I}_\mu$, then $v\in L^\infty(\R^N)$ and there
exists $C>0$ (independent of $\mu$) such that
$$|v|_\infty\leq C\left(\mu\|v\|^{q-2}\right)^{\frac{1}{2^*-q}}\|v\|.$$
\end{Lem}

\begin{proof}
For each $k>0$, we define
\begin{equation*}
v_k=\left\{
\begin{aligned}
& v,\quad\hbox{if}\ v\leq k,\\
& k,\quad\hbox{if}\ v\geq k.
\end{aligned}
\right.
\end{equation*}
Set $w_k=v v_k^{2(\gamma-1)}$ and $\wt{w}_k=v v_k^{\gamma-1}$, where $\gamma>1$ is to be
determined later. It follows from $\langle \wt{I}'_\mu(v), w_k\rangle=0$, \eqref{eq5.2}
and Lemma \ref{lem1} that
\begin{equation*}
\begin{array}{rl}
\di_{\R^N} v_k^{2(\gamma-1)}|\nabla v|^2 \ dx
&\leq \mu\di_{\R^N} \wt{h}(f(v))f'(v)v v_k^{2(\gamma-1)} \ dx \\[5mm]
&\leq C \mu \di_{\R^N}f^{q-1}(v)v v_k^{2(\gamma-1)} \ dx \\[5mm]
&\leq C \mu \di_{\R^N} v^q v_k^{2(\gamma-1)} \ dx \\[5mm]
&= C \mu \di_{\R^N} v^{q-2} \wt{w}_k^2 \ dx.
\end{array}
\end{equation*}
Combining this with Gagliardo-Nirenberg-Sobolev inequality yields
\begin{equation*}
\begin{array}{rl}
\left(\di_{\R^N} \wt{w}_k^{2^*} \ dx\right)^{\frac{2}{2^*}}
&\leq C\di_{\R^N} |\nabla \wt{w}_k|^2 \ dx\\[5mm]
&\leq C \di_{\R^N}\left(v_k^{2(\gamma-1)}|\nabla v|^2+(\gamma-1)^2v^2v_k^{2(\gamma-1)-2}|\nabla v_k|^2\right)\ dx \\[5mm]
&\leq C  \gamma^2 \di_{\R^N} v_k^{2(\gamma-1)}|\nabla v|^2\ dx \\[5mm]
&\leq C \gamma^2 \mu \di_{\R^N} v^{q-2} \wt{w}_k^2\ dx,
\end{array}
\end{equation*}
where we have used the facts that $v^2|\nabla v_k|^2\leq v_k^2|\nabla v|^2$ and
$1+(\gamma-1)^2<\gamma^2$ for $\gamma>1$. By H\"{o}lder inequality and Sobolev inequality,
\begin{equation*}
\begin{array}{rl}
\left(\di_{\R^N} \left(v v_k^{\gamma-1}\right)^{2^*}\ dx\right)^{\frac{2}{2^*}}
&\leq \left(\di_{\R^N}\wt{w}_k^{2^*}\ dx\right)^{\frac{2}{2^*}}
\leq C\gamma^2\mu\di_{\R^N} v^{q-2}\wt{w}_k^2\ dx\\[5mm]
&\leq C \gamma^2 \mu \left(\di_{\R^N} v^{2^*} \ dx\right)^{\frac{q-2}{2^*}}\left( \di_{\R^N} \wt{w}_k^{\frac{2\cdot2^*}{2^*-q+2}} \ dx \right)^\frac{2^*-q+2}{2^*}\\[5mm]
&\leq C\gamma^2 \mu \|v\|^{q-2}\left( \di_{\R^N}v^{\frac{2\gamma\cdot2^*}{2^*-q+2}}\ dx \right)^\frac{2^*-q+2}{2^*}.
\end{array}
\end{equation*}
Denote $\alpha_0=\frac{2\cdot2^*}{2^*-q+2}$. Choosing $\gamma=\frac{2^*-q+2}{2}$, we have $\frac{2\gamma\cdot2^*}{2^*-q+2}=2^*$ and so
$$\left(\di_{\R^N} \left(v v_k^{\gamma-1}\right)^{2^*}\ dx\right)^{\frac{2}{2^*}}\leq
C \gamma^2 \mu \|v\|^{q-2}|v|_{\gamma\alpha_0}^{2\gamma}.$$
Letting $k\to\infty$, by Fatou's lemma, we obtain
$$|v|_{\gamma\cdot 2^*}\leq \left(C\gamma^2 \mu \|v\|^{q-2}\right)^{\frac{1}{2\gamma}}|v|_{\gamma \alpha_0}.$$

For $m=0,1,\cdots$, set $\gamma_m=\gamma^{m+1}$. Repeating the above arguments for $\gamma_1$,
we have
\begin{equation*}
\begin{array}{rl}
|v|_{\gamma_1 \cdot 2^*}
&\leq \left(C\gamma_1^2 \mu \|v\|^{q-2}\right)^{\frac{1}{2\gamma_1}}|v|_{\gamma_1\alpha_0} \\[5mm]
&\leq \left(C\gamma_1^2 \mu\|v\|^{q-2}\right)^{\frac{1}{2\gamma_1}}
\left(C\gamma^2 \mu\|v\|^{q-2}\right)^{\frac{1}{2\gamma}}|v|_{\gamma\alpha_0}\\[5mm]
&=\left(C\mu\|v\|^{q-2}\right)^{\frac{1}{2\gamma_1}
+\frac{1}{2\gamma}}(\gamma)^{\frac{1}{\gamma}}(\gamma_1)^{\frac{1}{\gamma_1}}|v|_{2^*}.
\end{array}
\end{equation*}
By iteration, it follows that
$$|v|_{\gamma_m\cdot 2^*}\leq
\left(C\mu \|v\|^{q-2}\right)^{\frac{1}{2\gamma}\sum\limits_{i=0}^m\gamma^{-i}}
(\gamma)^{\frac{1}{\gamma}\sum\limits_{i=0}^m\gamma^{-i}}
(\gamma)^{\frac{1}{\gamma}\sum\limits_{i=0}^m i\gamma^{-i}}|v|_{2^*}.$$
Since $\gamma>1$, the series $\sum\limits_{i=0}^\infty \gamma^{-i}$ and
$\sum\limits_{i=0}^\infty i\gamma^{-i}$ are convergent. Taking $m\to\infty$, we get
$v\in L^\infty(\R^N)$ and
$$|v|_\infty\leq C\left(\mu\|v\|^{q-2}\right)^{\frac{1}{2^*-q}}\|v\|.$$
This completes the proof.
\end{proof}

\begin{Lem}\label{lem12}
Let $\mu>\frac{V_0}{4}$. If $v$ is a critical point of $\wt{I}_\mu$ with $\wt{I}_\mu(v)=d_\mu$,
then there exists $C>0$ (independent of $\mu$) such that
$$\|v\|^2\leq C\left(d_\mu+d_\mu^{\frac{2^*}{2}}+ \mu^{\frac{2^*-2}{q-2}}d_\mu^{\frac{2^*}{2}}\right).$$
\end{Lem}

\begin{proof}
By $\wt{I}_\mu(v)=d_\mu$, Lemma \ref{lem5} and $(V'_2)$, we obtain
\begin{equation} \label{eq5.6}
\di_{\R^N} |\nabla v|^2 \ dx\leq N d_\mu.
\end{equation}
Next we estimate the term $\int_{\R^N}V(x)f^2(v)\ dx$. It follows from
$\langle \wt{I}'_\mu(v),v\rangle=0$, Lemma \ref{lem1} and (\ref{eq5.3}) that
$$
\begin{array}{rl}
\df{1}{2}\di_{\R^N}V(x)f^2(v)\ dx
& \leq  \mu \di_{\R^N} \wt{h}(f(v)) f(v)\ dx\\[5mm]
& \leq \mu \va \di_{\R^N} f^2(v) \ dx + C\mu\va^{\frac{q-2^*}{q-2}} \di_{\R^N}|v|^{2^*}\ dx\\[5mm]
& \leq \df{\mu \va}{V_0}\di_{\R^N} V(x)f^2(v)\ dx +  C\mu\va^{\frac{q-2^*}{q-2}}
\left(\di_{\R^N}|\nabla v|^{2}\ dx  \right)^{\frac{2^*}{2}}
\end{array}
$$
Taking $\va=\frac{V_0} {4\mu}$ and using \eqref{eq5.6}, we have
\begin{equation}\label{eq5.7}
\di_{\R^N} V(x)f^2(v) \ dx \leq C \mu^{\frac{2^*-2}{q-2}} d_\mu^{\frac{2^*}{2}}.
\end{equation}
Then desired conclusion follows from \eqref{eq5.6}, \eqref{eq5.7} and the proof of Lemma \ref{lem2}.
\end{proof}

Now we are ready to prove Theorem \ref{thm3}.

\begin{proof}[Proof of Theorem \ref{thm3}]\  Observing that the functional $\wt{I}_\mu$ has
the Mountain Pass geometry, we can define
$$d_\mu=\inf\limits_{\gamma\in\Gamma}\max\limits_{t\in [0,1]}\wt{I}_\mu(\gamma(t))>0,$$
where $\Gamma=\{\gamma\in C([0,1],H_r^1(\R^N))\ |\ \gamma(0)=0,\ \wt{I}_\mu(\gamma(1))<0\}.$
Since $\wt{h}$ satisfies $(h_1)-(h_4)$, Theorem \ref{thm1} implies that there is a positive
critical point $v_\mu$ of $\wt{I}_\mu$ with $I_\mu(v_\mu)=d_\mu$.

Let $v_0\in C_0^\infty(\R^N)\setminus\{0\}$ be a nonnegative radially symmetric function such
that $\wt{I}_\mu(v_0)<0$. Then, by \eqref{eq5.4} with $T=|v_0|_\infty$ and properties of $f$,
we have
\begin{equation*}
\begin{array}{rl}
d_\mu & \leq \max\limits_{t\in [0,1]}\wt{I}_\mu(tv_0) \\[5mm]
& \leq \max\limits_{t\in [0,1]}\left(\df{t^2}{2}\di_{\R^N} \left(|\nabla v_0|^2+V_1 v_0^2 \right) \ dx
-\mu\di_{\R^N} \wt{H}(f(tv_0)) \ dx\right)\\[5mm]
& \leq \max\limits_{t\in [0,1]} \left(\df{t^2}{2}\di_{\R^N}
\left(|\nabla v_0|^2+V_1 v_0^2 \right) \ dx
-C\mu t^p\di_{\R^N} v_0^p \ dx\right)\\[5mm]
& \leq C \mu^{-\frac{2}{p-2}}.
\end{array}
\end{equation*}
Combining this with Lemmas \ref{lem11} and \ref{lem12}, we have
$$|v_\mu|_\infty\leq C \mu^{\frac{(p-q)[2^*(2^*-2)-2(p-2)]-2(2^*-p)(p-2)}{2(p-2)(q-2)(2^*-q)}},$$
for $\mu$ sufficiently large. Consequently, by \eqref{eq1.4}, we obtain a positive solution
$u_\mu=f(v_\mu)$ of problem \eqref{eq5.5} with $|u_\mu|_\infty\leq|v_\mu|_\infty<\delta$ for $\mu$
large enough. Then $u_\mu$ is a positive solution of problem \eqref{eq1.3}.
\end{proof}

\section{Generalized result}

\hspace*{\parindent}  In this section, we apply our methods to \eqref{eq1.2} with a general
nonlinearity of Berestycki and Lions type \cite{BL}. Throughout this section, we assume that
$V\in C^1(\R^N,\R)$ satisfies $(V_1),\ (V_2)$ and
\begin{enumerate}
\item[$(V_3)$] $\lim\limits_{|x|\to\infty} V(x)=V_0$.
\end{enumerate}
The nonlinearity $h$ satisfies $(h_1)$ and
\begin{enumerate}
\item[$(\bar{h}_2)$] $\lim\limits_{t\to+\infty} \df{h(t)}{t^{2\cdot 2^*-1}}=0$;
\item[$(\bar{h}_3)$] there exists $\zeta>0$ such that $H(\zeta)>\df{V_0}{2}\ \zeta^2$.
\end{enumerate}

\begin{Thm} \label{thm4}
Suppose that $V$ satisfies $(V_1),\ (V_2),\ (V_3)$ and $h$ satisfies $(h_1),\ (\bar{h}_2),\ (\bar{h}_3)$.
Then problem \eqref{eq1.2} has at least a positive solution.
\end{Thm}

Define $h_1=\max\{h,0\}$ and $h_2=\max\{-h,0\}$, then $h=h_1-h_2$. By $(\bar{h}_3)$, there
exists $\bar{u}\in H_r^1(\R^N)\cap L^\infty(\R^N)$ (see \cite{BL}) such that
$$\di_{\R^N}H_1(\bar{u})\ dx-\di_{\R^N}\left(H_2(\bar{u})+\df{V_0}{2}\ \bar{u}^2\right)\ dx=
\di_{\R^N}\left(H(\bar{u})-\df{V_0}{2}\ \bar{u}^2\right)\ dx>0,$$
where $H_i(t)=\int_0^t h_i(s) ds,\ i=1,2.$ Thus, for some $\bar{\lb}\in (0,1)$, we have
\begin{equation}\label{eq6.1}
\bar{\lb}\di_{\R^N}H_1(\bar{u})\ dx-\di_{\R^N}\left(H_2(\bar{u})+\df{V_0}{2}\ \bar{u}^2\right)\ dx>0.
\end{equation}

As in Section 3, we introduce a family of functionals
$$I_\lb(v)=\df{1}{2}\di_{\R^N} \left(|\nabla v|^2+V(x) f^2(v) \right)\ dx +\di_{\R^N} H_2(f(v))\ dx
-\lb\di_{\R^N} H_1(f(v))\ dx,$$
where $\lb\in[\bar{\lb},1]$. Set
$$A(v)=\frac{1}{2}\di_{\R^N} \left(|\nabla v|^2+V(x) f^2(v) \right)\ dx+\di_{\R^N} H_2(f(v))\ dx$$
and
$$B(v)=\di_{\R^N} H_1(f(v))\ dx,$$
then $I_\lb(v)=A(v)-\lb B(v)$. Similar to Lemma \ref{lem2}, we have

\begin{Lem} \label{lem13}
Assume that $(V_1),\ (V_3),\ (h_1),\ (\bar{h}_2)$ and $(\bar{h}_3)$ hold. Then\\
(1) $B(v)\geq 0$ for all $v\in H_r^1(\R^N)$;\\
(2) $A(v)\to \infty$ as $\|v\|\to\infty$;\\
(3) there exists $v_0\in H_r^1(\R^N)$, independent of $\lb$, such that $I_\lb(v_0)<0$ for all $\lb\in[\bar{\lb},1]$;\\
(4) for all $\lb\in[\bar{\lb},1]$, it holds
$$c_\lb=\inf_{\gamma\in\Gamma}\max_{t\in [0,1]}I_\lb(\gamma(t))>\max\{I_\lb(0),I_\lb(v_0)\},$$
\quad\ \ where $\Gamma=\{\gamma\in C([0,1],H_r^1(\R^N))\ |\ \gamma(0)=0,\ \gamma(1)=v_0\}.$
\end{Lem}

\begin{proof}
We only prove (3). Set
$$J_\lb(u)=\df{1}{2}\di_{\R^N}(|\nabla u|^2+V(x)u^2)\ dx + \di_{\R^N}u^2|\nabla u|^2\ dx
+\di_{\R^N} H_2(u)\ dx-\lb \di_{\R^N} H_1(u)\ dx.$$
For any $t>0$, we have
\begin{equation*}
\begin{array}{rl}
J_{\bar{\lb}}(\bar{u}(x/t))
& = \df{t^{N-2}}{2}\di_{\R^N} |\nabla \bar{u}|^2\ dx +\df{t^{N}}{2}\di_{\R^N} V(tx) \bar{u}^2\ dx\\[6mm]
&\qquad\quad+\ t^{N-2}\di_{\R^N} \bar{u}^2 |\nabla \bar{u}|^2\ dx+t^N\di_{\R^N} H_2(\bar{u})\ dx
-\bar{\lb}t^N \di_{\R^N} H_1(\bar{u})\ dx\\[6mm]
&= \df{t^{N-2}}{2}\left[\di_{\R^N} |\nabla \bar{u}|^2 \ dx
+ 2\di_{\R^N} \bar{u}^2 |\nabla \bar{u}|^2\ dx\right]\\[6mm]
&\qquad\quad+\ t^N\left[\df{1}{2}\di_{\R^N}V(tx)\bar{u}^2\ dx+\di_{\R^N} H_2(\bar{u})\ dx
-\bar{\lb}\di_{\R^N} H_1(\bar{u})\ dx\right].
\end{array}
\end{equation*}
By $(V_3)$ and \eqref{eq6.1}, it is easy to see that $J_{\bar{\lb}}(\bar{u}(x/t))<0$ for $t$
large. Then desired result follows.
\end{proof}

Arguing as the proof of Brezis-Lieb lemma \cite{BrL}, we have

\begin{Lem} \label{lem14}
Assume that $g\in C(\R,\R)$ satisfies
$$|g(t)|\leq C(|t|+|t|^{p-1}),\quad 2<p<2^*.$$
If $\{v_n\}$ is bounded in $H^1(\R^N)$ and  $v_n\to v\ \hbox{a.e. in}\ \R^N,$ then
$$\di_{\R^N}\left(G(v_n)-G(v_n-v)-G(v)\right)\ dx\to 0,\quad\hbox{as}\ n\to\infty,$$
where $G(t)=\int_0^t g(s) ds$.
\end{Lem}

The following lemma was due to Strauss \cite{St} (see also \cite{BL}).

\begin{Lem} \label{lem15}
Let $P$ and $Q$ be two continuous functions satisfying $P(t)/Q(t)\to 0$ as $t\to\infty$. If
$\{v_n\}$ is a sequence of measurable functions from $\R^N$ to $\R$ such that
$$\sup\limits_{n}\di_{\R^N}|Q(v_n)|\ dx<\infty$$
and $P(v_n)\to v\ \hbox{a.e. in}\ \R^N,$ then one has
$$\di_{\Omega}|P(v_n)-v|\ dx\to 0,\quad\hbox{as}\ n\to\infty,$$
for any bounded Borel set $\Omega$. Moreover, if one assumes also that $P(t)/Q(t)\to 0$ as $t\to 0$
and $\sup\limits_n |v_n(x)|\to 0$ as $|x|\to\infty$, then
$$\di_{\R^N}|P(v_n)-v|\ dx\to 0,\quad\hbox{as}\ n\to\infty.$$
\end{Lem}

\begin{Lem} \label{lem16}
Assume that $\{v_n\}\subset H_r^1(\R^N)$ is a bounded Palais-Smale sequence for $I_\lb$
satisfying $\lim\limits_{n\to\infty}I_\lb(v_n)=c_\lb$. Then, up to a subsequence, $\{v_n\}$
converges to a positive critical point $v_\lb$ of $I_\lb$ with $I_\lb(v_\lb)= c_\lb$.
\end{Lem}

\begin{proof}
First of all we may assume $v_n\rightharpoonup v_\lb\ \hbox{in}\ H_r^1(\R^N)$ and
$v_n\to v_\lb\ \hbox{a.e. in}\ \R^N.$ By Lebesgue dominate theorem and Lemma \ref{lem15},
we conclude that $I'_\lb(v_\lb)=0$. Set $w_n=v_n-v_\lb$. Using Lemma \ref{lem14} and Lemma
\ref{lem15} with $Q(t)=t^2+|t|^{2^*}$, we deduce
$$\langle I'_\lb(v_n), v_n \rangle-\langle I'_\lb(v_\lb), v_\lb \rangle-\langle I'_\lb(w_n), w_n \rangle=o(1)$$
and so $\langle I'_\lb(w_n), w_n \rangle=o(1)$. Therefore
\begin{equation*}
\begin{array}{rl}
&\quad\limsup\limits_{n\to\infty}\left(\di_{\R^N}|\nabla w_n|^2\ dx+\df{1}{2}\di_{\R^N}V(x)f^2(w_n)\ dx\right)\\[5mm]
&\leq \limsup\limits_{n\to\infty}\left(\di_{\R^N}|\nabla w_n|^2\ dx+\di_{\R^N}V(x)f(w_n)f'(w_n)w_n\ dx\right)\\[5mm]
&=\limsup\limits_{n\to\infty}\left(\lb\di_{\R^N}h_1(f(w_n))f'(w_n)w_n\ dx-\di_{\R^N}h_2(f(w_n))f'(w_n)w_n\ dx\right)
\\[5mm]
&=0.
\end{array}
\end{equation*}
Combining this with Lemma \ref{lem2}, we obtain $w_n\to 0$ in $H_r^1(\R^N)$. Consequently,
$v_n\to v_\lb$ in $H_r^1(\R^N)$. Hence $v_\lb$ is a nontrivial critical point of $I_\lb$ with
$I(v_\lb)=c_\lb$. A standard argument can show that $v_\lb>0$ in $\R^N$.
\end{proof}

\begin{proof}[Proof of Theorem \ref{thm4}]
It is similar to the proof of Theorem \ref{thm1}.
\end{proof}

\end{document}